\documentclass[11pt]{amsart}
\usepackage{amsmath,amsfonts,amsthm,amssymb,amscd}
\def\classification#1{\def\@class{#1}}
\classification{\null}

\DeclareMathAlphabet{\mathscr}{OT1}{rsfs}{n}{it}

\newcommand{\R}{{\mathbb R}}

\newcommand{\x}{\boldsymbol x}
\newcommand{\y}{\boldsymbol y}

\newtheorem{theorem}{Theorem}
\newtheorem{proposition}[theorem]{Proposition}

\newtheorem{lemma}[theorem]{Lemma}
\newtheorem{corollary}[theorem]{Corollary}

\theoremstyle{remark}
\newtheorem{remark}[theorem]{Remark}

\title{On the Minkowski distances and products of sum sets}
\author{Oliver Roche-Newton}
\address{Oliver Roche-Newton, Department of Mathematics, University of Bristol,
  Bristol BS8 1TW, United Kingdom}
\email{maorn@bristol.ac.uk}
\author{Misha Rudnev}
\address{Misha Rudnev, Department of Mathematics, University of Bristol,
  Bristol BS8 1TW, United Kingdom}
\email{m.rudnev@bristol.ac.uk}


\begin{document}

\begin{abstract}
Given two points $p,q$ in the real plane, the signed area of the rectangle with the diagonal $[pq]$ equals the square of the Minkowski distance between the points $p,q$.  We prove that $N>1$ points in
the Minkowski plane $\R^{1,1}$ generate $\Omega(\frac{N}{\log{N}})$ distinct distances, or all the distances are zero. The proof follows the lines of the Elekes/Sharir/Guth/Katz approach to the Erd\H
os distance problem, analysing the $3D$ incidence problem, arising by considering the action of the Minkowski isometry group $ISO^*(1,1)$.

The signature of the metric creates an obstacle  to applying the Guth/Katz incidence theorem to the $3D$ problem at hand, since one may encounter a high count of congruent line intervals, lying on null
lines, or ``light cones'', all these intervals having zero Minkowski length. In terms of the Guth/Katz theorem, its condition of the non-existence of ``rich planes'' generally gets violated. It turns out,
however, that one can efficiently identify and discount incidences, corresponding to null intervals and devise a counting stratagem, where the rich planes condition happens to be just ample enough for
the stratagem to succeed.

As a corollary we establish the following near-optimal sum-product type estimate for finite sets $A,B\subset \R$, with more than one element:
$$|(A\pm{B})\cdot{(A\pm{B})}|\gg{\frac{|A||B|}{\log{|A|}+\log{|B|}}}.$$

\end{abstract}

\maketitle

\begin{section}{Introduction}
The main result of this note resolves, up to logarithmic terms, a variant of the Erd\H os distance problem in the Minkowski plane $\R^{1,1}$, with the metric tensor
${\displaystyle\left(\begin{array}{rrr} 1& 0\\0&-1\end{array}\right).}$ This problem naturally arises if one has in mind a sum-product type estimate, of proving a lower bound on the size of the set
$(A+{B})\cdot(A+{B})$, where $A,B$ are finite sets of reals, with more than one element. As usual, $|\cdot|$ denotes the number of elements in a finite non-empty set, with more than one element, and
the sum-set $A+B$ is defined as
$$
A+B=\{a+b:\,a\in A, b\in B\}.
$$
The product, ratio, and difference sets, denoted respectively  as $A\cdot B$, $A:B$ (not dividing by zero if it is in $B$), $A-B$, are defined similarly. Throughout the paper, the standard notation
$\ll,\gg$ and, respectively, $O,\Omega$ is applied to positive quantities in the usual way. Say, $X\gg Y$ or $X=\Omega(Y)$ means that $X\geq cY$, for some absolute constant $c>0$.

We establish the following estimate:
\begin{equation}\label{mrs}|(A\pm{B})\cdot{(A\pm{B})}|\gg{\frac{|A||B|}{\log{|A|}+\log{|B|}}}.
\end{equation}

Note that the estimate \eqref{mrs} is sharp, up to the logarithmic terms. The example of the integer interval $B=A=[1,2,\ldots,|A|]$ shows that the right-hand side must indeed be $o(|A||B|)$.

On the other hand, the bound
\begin{equation}
|(A-A):(A-A)|\geq{|A|^2}, \label{ungarsp}
\end{equation}
has been known since the 1970s. The set in the left-hand side of (\ref{ungarsp}) is the set of non-vertical directions of the set of lines, incident to at least two points of the point set $A\times
A\subset \R^2$. The result of Ungar, \cite{U},  establishes, with the best possible constant, that ``$2N$ non-collinear points in the plane determine at least $2N$ distinct directions''. The fact
itself that the right-hand side of (\ref{ungarsp}) is $\Omega(|A|^2)$ had been known prior to \cite{U}: see the references, contained therein.

The proof of the lower bound for the number of distinct directions, determined by a non-collinear plane set of points seems to be using the topological properties of the real plane in a much more basic
way\footnote{Ungar considers orthogonal projections of the point set on a ``reference line'', which is being rotated. Upon every instance of the latter line finding itself perpendicular to some line,
determined by a pair of points in the set, the ordering in the set of the points' projections on the reference line changes. The problem of counting the directions, determined by $P$, gets thus
reformulated as a purely combinatorial one of counting moves, which reverse a permutation of $\{1,\ldots,N\}.$}  than the Szemer\'edi-Trotter type incidence theorems, which play the key role in the
proof herein, and we view the presence of the logarithmic term in (\ref{mrs}) as a testimony that this estimate is somewhat more subtle than (\ref{ungarsp}).

The estimate (\ref{mrs}) represents a weaker version of the well known and wide open Erd\H os-Szemer\'edi conjecture, \cite{ES},  claiming that for any $\varepsilon>0$, as $|A|\rightarrow\infty$,
\begin{equation}\label{ESC}
|A+A|+|A\cdot A|\geq |A|^{2-\varepsilon}.
\end{equation}
Here is a heuristic argument why the Erd\H os-Szemer\'edi conjecture is supposedly much harder than the question we address in this paper. Let us call a real function $f(x_1,\ldots,x_k)$ of several
variables an {\em expander} if with the values of each $x_i$ running over a finite set $A$, the range of $f$, denoted as $f(A,\ldots,A)$, has at least some $c_\epsilon |A|^{2-\varepsilon}$ elements, as
$|A|\rightarrow\infty$.  The smaller the number of variables $k$, the harder it is to establish that a particular $f(x_1,\ldots,x_k)$ is an expander. The first formula in (\ref{tim}) below gives a three-variable expander, see also the discussion following it.

There are no two-variable expanders known. For the best known expansion properties of a function $x_1(x_2+1)$ of two variables, see \cite{JRN} and the references contained therein. Also  $g$ is a strictly convex function of one variable, then one could conjecture, similar to (\ref{ESC}) that either $x_1+x_2$ or $f(x_1,x_2)=g(x_1)+g(x_2)$, should be an expander. The best known results about sums of convex sets are due to Schoen and Shkredov: see \cite{SS}, as well as \cite{LRN}.

The Erd\H os-Szemer\'edi conjecture claims that one of the two functions of two variables, $x_1+x_2$  or $x_1x_2$, evaluated on the same set, is an expander. The estimate (\ref{mrs}) states that a
single four-variable function $(x_1\pm x_2)(x_3\pm x_4)$ is an expander.

If $f$ has four variables, one can view $(x_1,x_2)$ and $(x_3, x_4)$ as two points in the plane, both lying in the point set $P=A\times A$. This enables one to deal with the problem, using geometric
combinatorics. This is exactly the case with the function $f(x_1,\ldots,x_4)=(x_1-x_3)(x_2- x_4)$, and in order to prove that this function is an expander, we will take advantage of the fact that it
equals the square of the Minkowski distance between the points $(x_1,x_2)$ and $(x_3, x_4)$ in $\R^{1,1}$. Hence, one is naturally led to a variant of the Erd\H os distance problem in $\R^{1,1}$.
Before turning to it, we remark that the Minkowski metric was also used by Hart et al., \cite{HIS}, as a vehicle to prove new at the time sum-product estimates over the finite field $\mathbb F_q$.

\medskip
The Erd\H os distance problem was originally formulated in \cite{E} -- for the case of the Euclidean distance -- in 1946. Applied verbatim to the Minkowski plane $\R^{1,1}$, it would claim that a
finite point set $P\subset \R^{1,1}$ determines, for any $\epsilon>0$, the set of squares of the Minkowski distances
\begin{equation}\label{MD}
\Delta_{1,1}(P):=\{(p_1-q_1)^2 - (p_2-q_2)^2: \;p,q\in P\},
\end{equation}
of cardinality at least $c_\epsilon |P|^{1-\epsilon}$.

In 2010, Guth and Katz, \cite{GK}, settled the Erd\H os distance problem (up to logarithmic terms) in the foundational case of the Euclidean plane. Stronger versions of the conjecture in $\R^2$ (e.g.
the pinned one -- proving the lower bound for the size of the set of distances from a ``typical'' point of $P$) remain open, as well as the full conjecture in higher dimension.

In the Minkowski plane, there is an immediate counterexample to the claim that the number of distinct distances is always ``almost'' $|P|.$ Let $P$ be supported on a ``light cone'', that is a single
line, whose slope equals $\pm1$. Then all the pair-wise Minkowski distances are zero. However, it turns out that this is the only counterexample. Our main result is the following theorem.

\begin{theorem}\label{minerd} Let $P\subset \R^{1,1}$ be such that $\Delta_{1,1}(P)\neq\{0\}$. Then, for some absolute $c$,
$$
|\Delta_{1,1}(P)| \geq c\frac{|P|}{\log|P|}.
$$
\end{theorem}

Let us further change the old coordinates $(x_1,x_2)$ in $\R^{1,1}$ to the new coordinates $(x_1',x_2')$ as follows:
\begin{equation}
x_1'= x_1 - x_2, \qquad x_2'= x_1+x_2, \label{coordchg}\end{equation} so that the metric becomes $d x_1^2- dx_2^2 = dx_1' dx_2'$. We further drop the prime indices for the new coordinates. These are
the new coordinates in $\R^{1,1}$ that we further use.

For two points $p=(p_1,p_2),\;q=(q_1,q_2)$ in $P$, the square of the Minkowski distance between them equals the {\em rectangular area}, $R(p,q)$, calculated as follows:
\begin{equation}\label{recta}R(p,q)=(q_1-p_1)(q_2-p_2).\end{equation}
Given two point sets $P,Q$ we define the set of rectangular areas as
$$R(P,Q):=\{R(p,q):p\in{P},\,q\in{Q}\}$$
and write $R(P)$ instead of $R(P,P)$. Note that $R(P)$ is just a different notation for the set of squares of Minkowski distances $\Delta_{1,1}(P)$, defined by (\ref{MD}).

\addtocounter{theorem}{-1}

\renewcommand{\thetheorem}{\arabic{theorem}'}

We restate Theorem \ref{minerd}  in terms of the rectangular areas, from now on let $N=|P|$. \begin{theorem} \label{theorem:main} Let $P$ be a set of $N$ points in the plane, such that not all points of $P$
lie on a single horizontal or vertical line. Then
\begin{equation}|R(P)|\gg{\frac{N}{\log{N}}}.\label{set}\end{equation}
\end{theorem}

\renewcommand{\thetheorem}{\arabic{theorem}}

The sum-product inequality (\ref{mrs}) becomes a direct corollary of Theorem \ref{theorem:main}, with $P=A\times B$. The theorem strictly speaking, entails the estimate (\ref{mrs}) immediately only in
the case of the $-$ signs. However, it generalises trivially to yield the same, up to a constant, bound for the set $R(P,-P),$ which enables one to use any combination of the $-$ and $+$ signs in the
estimate (\ref{mrs}).

The estimate \eqref{set} is sharp up to the logarithmic factor. As we've mentioned, if $P=[1,\ldots,|A|]^2$, then $\Delta_{1,1}(P)=R(P)$ is contained in the set of all pair-wise products of integers
between $-|A|+1$ and $|A|-1$, , whose cardinality is $o(|A|^2),$  being off from $|A|^2$ by a logarithmic factor. See \cite{F} for the precise asymptotics.

\medskip

The Guth/Katz proof of the Erd\H os distance conjecture is based on at least two fundamental new ideas. One, due in particular to Elekes and Sharir, \cite{ES}, is to interpret the plane problem about the distances
as a three-dimensional problem about incidences between straight lines in the Euclidean plane isometry group $SE_2(\R)$. The other is to introduce the polynomial method, fetching in particular the
polynomial version of the Ham Sandwich theorem of Stone and Tukey, \cite{StT}, to prove a generally sharp upper bound for the number of incidences in the latter incidence problem. Remarkably, the
latter bound applies to the corresponding incidence problem in $\R^3$ in the best possible way, without the further need to appeal to the underlying problem in $\R^2$.

More precisely, the estimate $\Omega(\frac{N}{\log{N}})$ for the number of distinct Euclidean distances, generated by a set of $N>1$ points in $\R^2$, arises -- after a dyadic summation and an
application of the Cauchy-Schwarz inequality -- from the following Szemer\'edi-Trotter type incidence theorem\footnote{We have amalgamated  the statements of Theorems 2.10, 2.11, and 4.5 in \cite{GK}
into one that suits our purposes. In particular, the condition (iii) of Theorem \ref{theorem:GKmain} is necessary only in the case $k=2$. The condition (i) enables one to subsume into (\ref{GKE}) the
``trivial'' estimate $|S_k|\gg\frac{N^2}{k}$, as $N^2$ lines can always be arranged as $N^2/k$ non-intersecting bushes of roughly $k$ lines in a bush. The condition (ii) is not mentioned explicitly in
Theorem 2.11 of \cite{GK}, being too ample for its purposes. However, it is spelled out in Theorem 4.5 therein.}.

\begin{theorem}[Guth-Katz]\label{theorem:GKmain}  Let $k\ge 2$. Consider a set of $N^2$ lines in $\mathbb{R}^3$ such that

(i) no more than $O(N)$ lines are concurrent,

(ii) no more than $O(N k)$ lines lie in a single plane,

(iii) no more than $O(N)$ lines lie in a single non-plane doubly ruled surface.

Let $S_k$ be the set of points in $\R^3$, incident to at least $k$ and at most $2k$ lines. Then
\begin{equation}\label{GKE} |S_k|\ll \frac{N^3}{k^2}.\end{equation}
\end{theorem}

For comparison, although it not used directly in this paper, we give an analogous statement of the Szemer\'edi-Trotter theorem, \cite{ST}.
\begin{theorem}[Szemer\'edi-Trotter]\label{theorem:ST} Let $k\ge 2$. Consider a set of $N^2$ lines in $\mathbb{R}^2$ , such that no more than $O(N)$ lines are concurrent. Let $S_k$ be the set of points
in $\R^2$, incident to at least $k$ lines. Then
$$|S_k|\ll \frac{N^4}{k^3}.
$$
\end{theorem}

Guth and Katz, see \cite{GK} and the references contained therein, give a detailed account of prior developments and applications of the main ideas behind their proof of the Erd\H os distance
conjecture. These ideas have been taken up to achieve some progress in problems dealing with sum and product sets. Iosevich and the authors, \cite{IRR}, showed that Theorem \ref{theorem:GKmain} enables
one to claim the following sum-product type  estimate:
\begin{equation}
|A\cdot A\;\pm\;A\cdot A|\gg \frac{|A|^{2}}{\log|A|}. \label{us}\end{equation} The estimate is based on the lower bound for the set of values of distinct cross products, generated by a point set in
$\R^2$, which follows from the Guth-Katz theorem after ``lifting'' the problem in the three-dimensional Lie group $SL_2(\R)$.

Jones, \cite{J}, considers the action of the group $PSL_2(\R)$ on the real projective line $\R P^1$, with the objective of establishing the minimum number of  cross-ratios determined by a finite set
$A\subset \R P^1$. He proves, among other things, that
\begin{equation}\begin{array}{rcl}
\left|\left\{\frac{(a-b)}{(b-c)}:\;a,b,c\in A\right\}\right|& \gg &\frac{|A|^2}{\log |A|},\\ \hfill \\ \left|\left\{\frac{(a-b)(c-d)}{(b-c)(a-d)}:\;a,b,c,d\in A\right\}\right|&\gg &{|A|^2}.\end{array}
\label{tim}\end{equation} However, the incidence problem in three dimensions that follows after lifting the original problem from $\R P^1$ to  $PSL_2(\R)$ turns out to be one between points and planes.
Jones could then take advantage of the ``pre Guth-Katz'' incidence theorems: the Szemer\'edi-Trotter theorem for the first estimate in (\ref{tim}) and a point/plane incidence theorem, based on the
works of Edelsbrunner et al, see e.g. \cite{EGS}, for the second estimate in (\ref{tim}).

The estimates (\ref{tim}) bear testimony to the fact that there is no guarantee that re-formulating the underlying arithmetic or geometric problem in terms of the symmetry group action would
necessarily lead one to near-optimal estimates without additional effort. Indeed, the first estimate in (\ref{tim}) corresponds to cross-ratios, pinned at $d=\infty$. It is clearly sharp, up to the
logarithmic term. On the other hand, the second estimate (\ref{tim}) is unlikely to be sharp: presumably the best possible one would have at least $c_\epsilon |A|^{3-\epsilon}$ in the right-hand side.
Similar observations have been made in an earlier work of Solymosi and Tardos, \cite{SoTa}, concerning the action of the complex M\"obius transformation group.

\medskip
Let us give a brief outline of the forthcoming proof of Theorem \ref{theorem:main}, which follows structurally the Guth-Katz approach to the Euclidean Erd\H{o}s distance problem. By the Cauchy-Schwarz inequality, in order to bound from below the number of distinct Minkowski distances,
it suffices to provide the upper bound on the number of pairs of congruent line intervals $[pq],\,[st]$ of equal Minkowski length. (We prefer the nomenclature {\em intervals} to {\em segments} as a
tribute to terminology in Special Theory of Relativity.)

Unless they lie on two distinct branches of the light cone, the intervals $[pq]$ and $[st]$ are congruent if and only if there is an isometry $\phi\in ISO^*(1,1)$, such that $\phi(p)=s$ and
$\phi(q)=t$. Here $ISO^*(1,1)$ denotes the subgroup, the connected component of the identity, of the three-dimensional Lie group $ISO(1,1)$ of the isometries of $\R^{1,1}$. Topologically,
$ISO^*(1,1)\cong \R^3$.

The set of the Minkowski isometries, taking some point $p=(p_1,p_2)\in P$ to some $s=(s_1,s_2)\in P$ can be parameterised as a straight line in $\R^3$. This can be easily seen, since after the
coordinate change (\ref{coordchg}) one has a faithful representation of $ISO^*(1,1)$ simply as
\begin{equation} \Phi = \left(
                 \begin{array}{ccc}
                   z & 0 & x \\
                   0 & z^{-1} & y \\
                   0 & 0 & 1 \\
                 \end{array}
               \right),\label{rep}\end{equation}
               with $z>0$, acting on three-dimensional vectors of the form $(a,b,1)$. The line $l_{ps}$ comes out after solving (which is possible for any $p,s$,  unless the two points lie on different
               coordinate axes)   $\Phi \tilde p = \tilde s$, where $\tilde p = (p_1,p_2,1)$ corresponds to the point $p=(p_1,p_2)$; similarly  $\tilde s = (s_1,s_2,1)$.
Hence, there is a one-to one correspondence between the set of pairs of congruent intervals and the set of pair-wise intersections of $N^2$ distinct lines in $\R^3$. (Strictly speaking, under this
consideration $[pq]$ and $[qp]$ are treated as distinct intervals, and there is no isometry, homotopic to the identity, that would reverse one into the other.)

It remains to verify the conditions of Theorem \ref{theorem:GKmain}. The conditions (i) and (iii) are satisfied in  the same way as they are in \cite{GK}. But the condition (ii) generally gets
violated. Nonetheless, it turns out that the condition is just ample enough to enable one to use a single instance of a ``divide and conquer'' strategy. We will show that each line from the family of
$N^2$ lines in question  may lie in one and only one ``rich plane'', and points of $S_k$ can be partitioned by whether or not at least half of the lines, incident to a point in $S_k$ lie each in some
rich plane.

In the latter case,  Theorem \ref{theorem:GKmain} applies, affecting only the constant hidden in the estimate (\ref{GKE}). In the former case, it turns out that pair-wise intersections inside rich
planes correspond to zero Minkowski distances and should be discounted. As a result, a rather crude argument of counting the remaining pair-wise intersections occurring in each rich plane individually
and then summing over rich planes suffices.

The rest of this paper is concerned with the proof of Theorem \ref{theorem:main}.
\end{section}
\begin{section}{Proof of Theorem \ref{theorem:main}}
\subsection{Rectangular quadruples}
Following \cite{ES} and \cite{GK}, the size $|R(P)|$ of the set of rectangular areas, or squares of Minkowski distances generated by $P$ will be estimated by the Cauchy-Schwarz inequality from the
upper bound on the number of pairs of intervals with equal Minkowski nonzero lengths, referred here as ``rectangular quadruples''.

For $x\in{R(P)}$, define its number of realisations
$$n(x):=|\{(p,q)\in{P\times{P}}:R(p,q)=x\}|.$$
Observe that
\begin{equation}
N^2=\sum_{x\in{R(P)}}n(x)=\sum_{x\neq 0}n(x)+|\{(p,q)\in{P\times{P}}:R(p,q)=0\}|. \label{CS0}
\end{equation}
Let us further use the notation $R^*(P)= R(P)\setminus\{0\}$.

We can assume from now on that a single vertical or horizontal line does not support more than $cN$ points of $P$, for otherwise, as long as there is a point of $P$ not supported on the latter line --
a circumstance which can be assumed thanks to the hypotheses of Theorem \ref{theorem:main} -- we have $|R(P)|\gg N$. 

It follows from (\ref{CS0}) that
\begin{equation}
N^2\ll \sum_{x\in R^*(P)}n(x). \label{CS}
\end{equation}

Indeed, if the second term in (\ref{CS0}) were to exceed $2cN^2$, there would be a $p\in P$, such that $|\{q\in P:\,R(p,q)=0\}|\geq 2cN$, hence a vertical or horizontal line supporting at least $cN$
(but not all) points of $P$.

Hence, by the Cauchy-Schwarz inequality,
\begin{equation}
N^4\ll{\left(\sum_{x\in R^*(P)}n(x)^2\right)|R(P)|}. \label{CS2}
\end{equation}
A quadruple $(p,q,s,t)\in{P}\times P\times P\times P$ is defined to be a \textit{rectangular quadruple} if
\begin{equation}
R(p,q)=R(s,t)\neq 0. \label{energy}
\end{equation}
Note that the quantity, denoted further as
\begin{equation}{\mathcal Q}=\sum_{x\in R^*(P)}n(x)^2\label{qucal}
\end{equation} is the number of rectangular
quadruples.

In other words, ${\mathcal Q}$ is the number of Minkowski-congruent pairs of intervals defined by $P$, with non-zero Minkowski length.

\medskip
Theorem \ref{theorem:main} follows from the next Proposition, by combining it with \eqref{CS2}.

\begin{proposition} \label{incidences} For a set $P$ of $N$ points in the plane, the number of
rectangular quadruples
\begin{equation}\label{prove}{\mathcal Q} \ll N^3\log{N}. \end{equation}
\end{proposition}

The rest of the paper proves Proposition \ref{incidences}.

\subsection{Isometries}
Following \cite{ES} and \cite{GK}, the quantity $\mathcal Q$ will be related to an incidence problem in the three-dimensional Lie group $ISO^*(1,1)$ of Minkowski isometries, that is affine
transformations of the plane which preserve rectangular area and are homotopic to the identity. The exposition in this section is elementary and self-contained.

\subsubsection*{Notation:}
$\,$\newline \indent 1) We use the notation $\R_{+}$ for the multiplicative group of  positive reals.

2) For $x,y\in{\mathbb{R}}$, the translation by $(x,y)$ will be denoted by $T_{(x,y)}$. So $T_{(x,y)}(a,b)=(x+a,y+b)$.

3) For $z\in \R_{+}$, $H_z$ denotes the rectangular area preserving dilation defined by $H_z(a,b)=(za,\frac{b}{z})$. In the old coordinates, see (\ref{coordchg}), $H_z$ would be a Lorentz
transformation, a hyperbolic rotation by the angle $\ln z$.

\medskip
An affine map $\phi:\mathbb{R}^2\rightarrow{\mathbb{R}^2}$ is said to be \textit{rectangle preserving} if it is of the form $\phi=T_{(x,y)}\circ{H_z}$, for some triple $(x,y,z)\in
\R\times\R\times\R_{+}\cong \R^3$. In other words, if $p=(p_1,p_2)$, then $\phi(p) = (x+ z p_1,y+ z^{-1}p_2)$. It is easy to check that rectangle preserving maps form a subgroup, further denoted as
$G_R$, of the group $Aff(2,\R)$ of the affine transformations of the plane. In particular, if $\phi=T_{(x,y)}\circ{H_z}$, then $\phi^{-1} = T_{(-x z^{-1}, -yz)}\circ S_{z^{-1}}.$ Since $z$ is positive,
any such $\phi$ is homotopic to the identity. The above group $G_R$ of rectangle preserving transformation is known in literature as $ISO^*(1,1)$, and the statement that (\ref{rep}) is its faithful
representation summarises the content of this passage.

Let us take a moment to  verify explicitly that such transformations $\phi$ preserve rectangular area. With $\phi=T_{(x,y)}\circ{H_z}$, we have $\phi(p_1,p_2)=(x+zp_1,y+z^{-1}p_2)$ and
 $\phi(q_1,q_2)=(x+zq_1,y+z^{-1}q_2).$ Therefore,

\begin{equation}\label{prs}
\begin{aligned}
R(\phi(p),\phi(q))&=(q_1z+x-(p_1z+x))(y+{q_2}{z^{-1}}-(y+{p_2}{z^{-1}}))
\\&=z(q_1-p_1)z^{-1}(q_2-p_2)
\\&=R(p,q).
\end{aligned}\end{equation}

The next statement is that if two line intervals have equal and non-zero non-zero Minkowski length, there is a unique isometry in $G_R$ (that is homotopic to the identity), taking one interval to the
other.
\begin{lemma} \label{lmm} Let $p,q,s,t\in{P}$ such that $R(p,q)\neq{0}$. Then $(p,q,s,t)$ is a rectangular
quadruple if and only if there exists a unique $\phi\in{G_R}$ such that $\phi(p)$ equals $s$ or $t$ and, respectively,  $\phi(q)$ equals $t$ or $s$.
\end{lemma}

\begin{proof}By (\ref{prs}), it remains to prove the necessity. By definition of rectangular quadruples,
$R(p,q)=R(s,t)\neq 0.$ Write $p=(p_1,p_2)$ and do the same for $q,s$ and $t$. Note that $p_1\neq q_1$, $s_1\neq t_1$.

Consider the map $\phi=T_{(x,y)}\circ{H_z}$, where
\begin{align*}
z&=\frac{t_1-s_1}{q_1-p_1},
\\x&=s_1- z p_1,
\\y&=s_2- z^{-1} p_2.
\end{align*}
Suppose, $z>0$, otherwise swap $s$ and $t$. Then the map $\phi$ is an element of $G_R$, and has the property that $\phi(p)=s,\,\phi(q)=t$.

This map is rectangle preserving, so all that is required is to check that $\phi(p_1,p_2)=(s_1,s_2)$ and $\phi(q_1,q_2)=(t_1,t_2)$. Indeed,
$$
\phi(p_1,p_2)=(s_1- zp_1 + z p_1 , s_2-z^{-1} p_2+z^{-1} p_2)=(s_1,s_2),$$ as required. Also, since $(q_1-p_1)(q_2-p_2)=(t_1-s_1)(t_2-s_2)$,
\begin{align*}
\phi(q_1,q_2) &= (s_1 + z(q_1 - p_1) , s_2+z^{-1}(q_2- p_2))\\ &=\left(t_1, s_2 + z^{-1}\frac{(t_1-s_1)(t_2-s_2)}{q_1-p_1}\right)\\ &=(t_1,t_2).
\end{align*}
Uniqueness follows by construction, which was tantamount to solving for $x,y,z$ the over-defined system of equations
$$
\left(
                 \begin{array}{ccc}
                   z & 0 & x \\
                   0 & z^{-1} & y \\
                 \end{array}
               \right) \left(
                         \begin{array}{c}
                           p_1 \\
                           p_2 \\
                           1
                         \end{array}
                       \right) = \left(
                                 \begin{array}{c}
                                   s_1 \\
                                   s_2 \\
                                 \end{array}
                               \right),\qquad  \left(
                 \begin{array}{ccc}
                   z & 0 & x \\
                   0 & z^{-1} & y \\
                 \end{array}
               \right) \left(
                         \begin{array}{c}
                           q_1 \\
                           q_2 \\
                           1
                         \end{array}
                       \right) = \left(
                                 \begin{array}{c}
                                   t_1 \\
                                   t_2 \\
                                 \end{array}
                               \right).
$$
This solution is feasible and unique, provided that $R(p,q)=R(s,t)\neq 0$.
\end{proof}

\begin{remark} Note that the case of zero Minkowski distance shall be indeed excluded. If $p$ and
$q$ share the same abscissa, while $s$ and $t$ share the  same ordinate, then one has $0=R(p,q)=R(s,t)$, but there is no $\phi\in G_R$ such that $\phi(p)=s$ and $\phi(q)=t$. One may not swap time and
space-like variables in the Pseudoeuclidean metric.  \end{remark}

In line with the statement of Lemma \ref{lmm}, let us narrow the definition of rectangular quadruples by from now on regarding $p,q,s,t$ as a rectangular quadruple if there is $\phi\in G_R$, such that
$\phi(p)=s$ and $\phi(q)=t$. This reduces the count of rectangular quadruples defined by the condition (\ref{energy}) precisely by the factor of $2$.

\medskip Let $p,s\in P$. Define
\begin{equation}\label{transform} L_{ps}:=\{\phi\in G_R:\phi(p)=s.\}\end{equation}

We shall shortly identify a coordinate system in $\R^3$, in which $L_{ps}$ will become a straight line. For now let us summarise the situation as follows.

\begin{corollary} Let $p,q,s,t\in{P}$ such that $R(p,q)\neq{0}$. Then $(p,q,s,t)$ is a rectangular
quadruple if and only if $L_{ps}\cap{L_{qt}}\neq{\emptyset}$.
\end{corollary}

Observe that for any quadruple $(p,q,s,t)$ of points from $\mathbb{R}^2$,
\begin{equation}
|L_{ps}\cap{L_{qt}}|>1\;\Leftrightarrow\;{(p,s)=(q,t)}. \label{intersections}
\end{equation}
Indeed, the condition $|L_{ps}\cap{L_{qt}}|>1$ means, in particular, that the system of linear equations $s_1= x+zp_1,\,t_1=x+zq_1$ has more than one solution $(x,z)$, which may only happen if
$q_1=p_1$, $s_1=t_1$. Similarly $|L_{ps}\cap{L_{qt}}|>1$ implies that $q_2=p_2$, $s_2=t_2$.

Let $\Psi: G_R\rightarrow \R\times\R\times \R_{+}$ be the map defined by
$$\Psi(T_{(x,y)}\circ{H_z}):=(x,yz,z).$$
Clearly, $\Psi$ is a bijection.

\begin{lemma} \label{lns} Let $p,s\in P$. Then $l_{ps}=\Psi(L_{ps})$ parameterises an open straight line interval in
$\mathbb{R}^3.$
\end{lemma}
\begin{proof}
Write $p=(p_1,p_2)$ and $s=(s_1,s_2)$. Then
$$
L_{ps} = \{(\phi = T_{x,y}\circ H_z):\, s_1=x+zp_1,s_2=y+z^{-1}p_2\}.
$$
Hence, for $z\in \R_{+}$,
$$x=s_1-p_1z, \qquad yz = -p_2 + s_2z, \qquad z=z,$$
which parameterises a straight line in the coordinates $(x, yz,z)$ in the half-space $z>0$.
\end{proof}

Denote
\begin{equation}L=\{l_{ps}:\,p,s\in{P}\}.\label{ourlines}\end{equation}
So $L$ is a set of $N^2$ lines in the half-space $z>0$ in $\R^3$ (we further just say, in $\R^3$), whose elements $l_{ps}=\Psi(L_{ps})$, after an obvious change of signs to symmetrize things, are
\begin{equation}\label{lps}
l_{ps}=\{(x,y,z): x= s_1-p_1z,\,y=p_2-s_2z,\,z>0\}.\end{equation}

\begin{remark} {\rm Note that the set $L_{ps}$ is the right coset in $G_R$ of the translation that
takes $p$ to $s$ by the subgroup stabilising $s$, which is the $SO^*(1,1)$-type subgroup of hyperbolic rotations, homotopic to the identity, with $s$ as the origin. Topologically $SO^*(1,1)\cong \R$, the
equation (\ref{lps}) providing an explicit parameterisation for the corresponding line $l_{ps}$ in the open half-space $z>0$.}\end{remark}

We summarise the argument so far as follows.
\begin{proposition} \label{keyc} Let $p,q,s,t\in{P}$, such that $R(p,q)\neq 0$. Then $(p,q,s,t)$ is a
rectangular quadruple only if $l_{ps}\cap l_{qt}\neq{\emptyset}$. Thus if $S$ is a set of points in $\R\times \R\times \R_{+}$ where pairs of distinct lines intersect, and for $\sigma \in S$,
$n(\sigma)$ denotes the number of pairs of distinct lines from $L$ intersecting at $\sigma$, then
\begin{equation}\label{rough}
\mathcal Q \leq \sum_{\sigma\in S} n(\sigma).\end{equation}
\end{proposition}

Proposition \ref{keyc} means that the problem of counting the number of non-zero solutions to \eqref{energy} to prove Proposition \ref{incidences} has been reduced to a question about the number of
pairwise crossings of a set of $N^2$ lines in $\mathbb{R}^3$. The count (\ref{rough}) may be refined a step further, as not all pairwise intersections should be counted: an intersection $l_{ps}\cap
l_{qt}$ yields a rectangular quadruple only if the points $p,q$ do not share a common coordinate. However, it cannot be established by looking at the line $l_{ps}$ alone which incidences of it with
other lines should be counted and which should not.

We now have all the ingredients to prove Proposition \ref{incidences}.

\subsection{Proof of Proposition \ref{incidences}}
Let us check the assumptions (i)--(iii) of Theorem \ref{theorem:GKmain} applied to the set $L$ of lines. The second assumption is the one concerning almost all of the ensuing discussion, so it is dealt
with last of the three.

\medskip
(i) No more than $N$ lines can be concurrent. This is indeed true. If more than $N$ lines $l_{ps}$ intersect at a point, there is some $p\in P$ and $\phi\in G_R$, such that $\phi(p)=s$, and
$\phi(p)=s'$ for some distinct points $s$ and $s'$. This is clearly a contradiction.

\medskip
(iii) No more than $O(N)$ lines can lie in a single non-plane doubly ruled surface $\Sigma$ (of which there are only two types: the one-sheeted hyperboloid or the hyperbolic paraboloid). The analysis
of Proposition 2.8 and Lemma 2.9 from \cite{GK} can be copied with straightforward modifications to show that this condition holds.

In brief: let us eliminate the variables $s_1,s_2$ from the explicit equations for the lines $l_{ps}$ in (\ref{lps}) and consider a family of lines $L_p$, foliating $G_R$. Each line $l_{ps}\in L_p$
corresponds to the one-dimensional family of isometries taking the point $p$ to some $s$, where $s$ is being viewed as a continuous variable in $\R^2$. Direction vectors to these lines form a vector
field, whose coordinates at the point $(x,y,z)$ are $(p_1,(y-p_2)z^{-1},1)$. Without loss of generality set $z=1$. As a line in $L_p$ pierces the $z=1$ plane at a point $(x,y,1)$, the direction vector
to the line is $d_p(x,y) =(p_1,y-p_2,1)$. Let us make the following claim. \begin{quotation} {\em Claim: if more than a finite number $K=2$ of lines from a single ruling of $\Sigma$  lie in $L_p$, then all the lines in the ruling must lie in
$L_p$.}\end{quotation} Assume the claim for a moment.  It follows that $\Sigma$ cannot contain more than $2KN$ lines of the family $L$. Indeed, otherwise one ruling of $\Sigma$ would have more than $KN$ lines of the family $L$, hence more than $K$ lines $l_{ps}$, for some $p$. Therefore, the ruling itself would be contained in the family $L_p$, and thus contain no lines $l_{qt}$ for $q\neq p$. Which means, the ruling contains at most $N$ lines of the family $L$, which is a contradiction. Hence, the condition (iii) of
Theorem \ref{theorem:GKmain} will be verified if the claim is shown to be true.

Let us verify the claim. To have a line from a ruling of $\Sigma$ belong to the family $L_p$, the lines from the ruling must intersect the plane $z=1$ transversely. Hence $\Sigma$ intersects the $z=1$ plane along an
irreducible quadratic curve $f(x,y)=0.$ (Otherwise the intersection occur along a pair of lines from two distinct rulings.) Components of the direction vectors of the lines in a single ruling of
$\Sigma$, evaluated at $z=1$, are linear functions of $(x,y)$. (This follows from the explicit equations for the one-sheeted hyperboloid or the hyperbolic paraboloid, or more generally, the fact that
both are degree $2$ irreducible surfaces.) In order to be able to compare  the latter direction vectors with $d_p$, one should normalise their third component of by $1$. This results in
linear-fractional quantities $d_1(x,y)$ and $d_2(x,y)$. So, for the lines of the ruling in question, the direction vectors at $z=1$ are $(d_1(x,y),d_2(x,y),1)$.

Now, assuming that more than $K$ lines in the ruling belong to $L_p$  is tantamount to the claim that the system of equations
\begin{equation}
\left\{
\begin{array}{lcl}
    f(x,y)&=&0, \\
    d_1(x,y)&=&p_1, \\
    d_2(x,y)&=&y-p_2.
  \end{array}\right.
\label{regeq}\end{equation} has more than $K$ distinct solutions $(x,y)$. Note that the second equation is linear in $(x,y)$. Hence the first pair of equations cannot have more than two distinct
solutions $(x,y)$, unless $d_1(x,y)=p_1=const.$ The fact that it is a constant, once again from the explicit equations for the explicit equations for the one-sheeted hyperboloid implies that the
function $d_2(x,y)$ is a linear function\footnote{Without using this fact, namely assuming that the function $d_2(x,y)$ is linear-fractional rather than linear, one can nonetheless conclude that the
first and third equations together in (\ref{regeq}) can have no more than four distinct solutions $(x,y)$, unless $f$, which is irreducible,  is a factor of the quadratic polynomial arising from the
third equation. This renders the same conclusion with $K=4$.} in $(x,y)$. Once again, there cannot be more than two distinct solutions $(x,y)$ of the system of equation (\ref{regeq}), unless
$d_2(x,y)=y-p_2.$ In the latter case, all the lines in the ruling of $\Sigma$ are in $L_p$.

\medskip
(ii) We now come to analyse how many of the lines of $L$ can lie in a single plane. Let $\pi$ be some plane in $\mathbb{R}^3$, with the equation $\alpha{x}+\beta{y}+\gamma{z}=\delta$. Consider the
intersection of $\pi$ with some line $l_{ps}$ in $L$. This intersection is the set of all $(x,y,z),$ with $z>0$, that satisfy the following system of equations:
\begin{equation}
\left\{ \begin{array}{llllllll}
                                          x && &+&p_1z &=&s_1,\\
                                           & & y &+&s_2z  &=&p_2,\\
                                          \alpha{x}&+&\beta{y}&+&\gamma{z} &=&\delta.
                                        \end{array}\right.\label{lineq1}\end{equation}

The line $l_{pq}$ lies in the plane $\pi$ only if \eqref{lineq1} has infinitely many solutions. This happens only if the system of equations is degenerate, i.e., the rows of its coefficients are
linearly dependent. This may not occur if $\alpha=\beta=0$.

If the plane $\pi$ is such that both $\alpha,\beta\neq 0,$ then  the linear dependence of the rows of the coefficients of (\ref{lineq1}) would mean that the quartet $(p_1,p_2,s_1,s_2)$ is such that
$$
\alpha p_1+\beta s_2 = \gamma, \qquad \alpha s_1+\beta p_2 =\delta.
$$
Thus given $p=(p_1,p_2)$, there is only one possible $s=(s_1,s_2)$ to satisfy the above equations, and hence the plane $\pi$ may not contain more than $N$ lines from $L$.

Suppose now that the plane $\pi$ is such that $\beta=0$. Then we can fix $\alpha=1$, and the linear dependence of the rows of the coefficients in (\ref{lineq1}) means that $p_1=\gamma$, $s_1=\delta$.

The same argument applies to the case $\alpha=0$. Hence we have shown the following.

 \begin{lemma} \label{rich} The assumption (ii) of Theorem \ref{theorem:GKmain} may fail for some plane
 $\pi$ (and some $k$) only if
$\pi$ has equation

 (i) either $x+p_1z=s_1$  where  the quantities $(p_1, s_1)$ are such that there are $\gg N$ pairs
 of points $(p,s)$, with abscissae $p_1,s_1$, respectively,

 (ii) or $y+s_2z=p_2$, where the quantities $(p_2, s_2)$ are such that there are $\gg N$ pairs of
 points $(p,s)$, with ordinates $p_2,s_2$, respectively.
\end{lemma}

\begin{remark} {\rm Note that in the special case $P=A\times{A}$ the above scenario in Lemma \ref{rich} cannot occur, and so all the conditions of Theorem \ref{theorem:GKmain} are satisfied.
Therefore,
the ensuing analysis is not needed in order to prove the sum-product type estimates \eqref{mrs} in the special case $A=B$, which follows right at this point, by Theorem \ref{theorem:GKmain}. Indeed, if
the assumption (ii) of Theorem \ref{theorem:GKmain} is satisfied, one has, after a dyadic summation in $k$:
\begin{equation}\mathcal Q\ll \sum_{k=2^j\leq N} k^2|S_k|\ll {N^3}\,{\log N}.\label{dysum}\end{equation}
However, the scenario in Lemma \ref{rich} already takes place if $P=A\times{B}$, with, say $|B|=o(|A|)$.} \end{remark}

In view of Lemma \ref{rich}, let us label an abscissa or ordinate (in the plane, where the set $P$ lives and the group $G_R$ acts) as {\em rich} if there are $>2\sqrt{N}$ points of $P$ having this
abscissa or ordinate; otherwise it is labeled as {\em poor}.

Then, since strictly less than each fourth point of $P$ may have both coordinates labeled as rich, at least one of the following two cases always occurs.

{\sf Case 1.} There is $P'\subseteq P$, containing at least $25\%$ of points of $P$, with the property that one specific coordinate is rich and the other coordinate is poor.

{\sf Case 2.} There is $P'\subseteq P$, containing at least $25\%$ of points of $P$ are such that both coordinates are poor.

If Case 2 occurs, then we replace $P$ by $P'$, restricting all the above constructs, such as $\mathcal Q$, $L$ to it. It follows from Lemma \ref{rich} that the assumption (ii) of Theorem
\ref{theorem:GKmain} is now satisfied. Theorem \ref{theorem:main} follows.

\medskip
Hence, throughout the rest of the proof it will be assumed that $P$ is such that {\em all} its abscissae are rich and {\em all} ordinates are poor. That is, assumption  (ii) of Theorem
\ref{theorem:GKmain} may fail, for some value of the parameter $k$, in a plane $\pi$ if and only if the equation of $\pi$ is $x+p_1z=s_1$, where $p_1,s_1$ are some two abscissae from the projection of
$P$ on the $x$-axis. Let us refer to these planes as rich planes. Thus, each line $l_{ps}$ defined in (\ref{lps}) lies in one and only one rich plane.

We now make an important observation. Whenever we have, for some rich plane $\pi$, that $\pi$ contains two distinct lines $l_{ps}$ and $l_{qt}$, this means that $p_1=q_1$ and $s_1=t_1$, so
$R(p,q)=R(s,t)=0$. I.e., the intersection of $l_{ps}$ and $l_{qt}$ does not contribute to the number of rectangular quadruples $\mathcal Q$, since it was defined relative to nonzero Minkowski
distances only. Thus we can refine (\ref{rough}) as follows:

\begin{equation}\label{ref}
\mathcal Q\ll \sum_{\sigma\in S} n^*(\sigma),
\end{equation}
where $n^*(\sigma)$ is the number of pairs of distinct lines incident to $\sigma$, but not lying in the same rich plane.

To complete the proof of Proposition \ref{incidences}, let $k$ be a dyadic integer, such that $2\leq k=2^j\leq \lceil \log_2 N \rceil$, and let $S_k$ be as in Theorem \ref{theorem:GKmain}, that is the
subset of $S$, containing all points $\sigma$, incident to a number of distinct lines from $L$ in the interval $[k,2k].$ It will be further shown that for every such $k$, one has

\begin{equation}\label{refk}
\sum_{\sigma\in S_k} n^*(\sigma) \ll N^3.
\end{equation}
Summing over the dyadic values $2\leq k=2^j\leq \lceil \log_2 N \rceil$ would then establish the desired estimate (\ref{prove}).

In order to proceed, we have to refine the notion of what rich is, relative to the value of $k$ in Theorem \ref{theorem:GKmain}. Given a dyadic $k:\,2\leq k\leq N$, we call a pair of abscissae in $\R^2$, and hence
the corresponding plane $\pi$ in $\R^3$, $k-$rich if $\pi$ supports more than $Nk$ lines from $L$. We also partition  $L$ into $k-$rich and $k-$poor lines by whether or not a particular line lies in
some $k-$rich plane $\pi$. Note that there may be only one such $\pi$ for each line. We now write $S_k = S_k^p \cup S_k^r$, where
$$\begin{array}{c} S_k^p=\{\sigma \in S_k:\mbox{  at least } \frac{k}{2} \mbox{ lines, incident to }\sigma \mbox{ are $k-$poor}\},\\ \hfill \\
S_k^r=\{\sigma \in S_k:\mbox{  at least } \frac{k}{2} \mbox{ lines, incident to }\sigma \mbox{ are $k-$rich}\}.
\end{array}$$

Theorem \ref{theorem:GKmain} now applies to the set $S_k^p$, since the set of $k-$poor lines satisfies the assumption (ii) of the theorem. Hence, the set $S_k^p$ satisfies the size estimate
(\ref{GKE}), and its contribution to the quantity $\mathcal Q$ is
$$
\sum_{\sigma\in S_k^p} n^*(\sigma) \ll |S_k^p|k^2 \ll N^3,
$$
conforming with (\ref{refk}).

\medskip
It remains to show that the set $S_k^r$ also contributes to the quantity $\mathcal Q$ at most $O(N^3)$, regardless of the (dyadic) value of $k$. Since each line may lie in at most one $k-$rich plane
$\pi$, the number of such planes $m$ is, by definition of  $k-$richness, at most $\frac{N}{k}$. Let $\pi_1,\ldots,\pi_m$ be the $k-$rich planes. For $i=1\ldots, m$, let $X_i=S_k^r\cap \pi_i$. The
subsets $X_i$ are not necessarily disjoint, but the crude estimate of summing the contribution of each of $X_i$ into the quantity $\mathcal Q$ suffices. Let us partition $X_i$ into two sets $X_i^\perp$
and $X_i^\|$ by whether or not at least half of the lines incident to the point $\sigma \in X_i$ are transverse to $\pi_i$ or not. Since there are at most $N^2$ transverse lines and at least
$\frac{k}{2}$ of them are incident to each point of the set $X_i^\perp$,
$$|X_i^\perp|\leq 2\frac{N^2}{k}.$$ Thus
\begin{equation}
\sum_{\sigma\in X_i^\perp} n^*(\sigma)\leq 4k^2|X_i^\perp|\leq 8kN^2. \label{perpcon}\end{equation}

\medskip
It remains to estimate the contribution of the sets $X_i^\|$ to the quantity $\mathcal Q$. Given $\sigma\in X_i^\|$, let $k^\perp(\sigma)\ll k$ be the number of lines of $L$  transverse to $\pi_i$ and
incident to $\sigma.$ Let $k^\|(\sigma)\geq \frac{k}{2}$ be the number of lines of $L$ incident to $\sigma$ and lying in the plane $\pi_i$.

Then, as pair-wise intersections inside $\pi_i$ do not contribute to the quantity $\mathcal Q$,
\begin{equation}\label{parcon}
\sum_{\sigma\in X_i^\|} n^*(\sigma)\ll \sum_{\sigma\in X_i^\|} (k^\perp(\sigma)^2 + k^\|(\sigma)k^\perp(\sigma))\ll k \sum_{\sigma\in X_i^\|} k^\perp(\sigma) \leq kN^2.
\end{equation}

Recalling that the number of $k-$rich planes is at most $\frac{N}{k}$ leads one to conclude from the estimates (\ref{perpcon}) and (\ref{parcon}) that
$$
\sum_{\sigma\in S_k^r} n^*(\sigma) \ll \frac{N}{k} \cdot(kN^2) = N^3.
$$
This establishes the estimate (\ref{refk}), and completes the proof of Proposition \ref{incidences} and Theorem \ref{theorem:main}.

\qed
\end{section}

\end{document}